\documentclass{amsart}

\usepackage{graphicx}
\usepackage{url}
\theoremstyle{plain}
\newtheorem{Prop}{Proposition}
\newtheorem{Remark}{Remark}
\newtheorem{Theo}{Theorem}
\newtheorem{Lem}{Lemma}
\newtheorem{Cor}{Corollary}

\allowdisplaybreaks

\newcommand{\Si}{\operatorname{Si}}
\newcommand{\Ci}{\operatorname{Ci}}

\begin{document}

\title{Fractional moments}

\author{\'Oscar Ciaurri}
\address{Departamento de Matem\'aticas y Computaci\'on,
Universidad de La Rioja, Complejo Cient\'{\i}fico-Tecnol\'ogico,
Calle Madre de Dios 53, 26006 Logro\~no, Spain}
\email{oscar.ciaurri@unirioja.es}

\thanks{The author was supported by grant  PGC2018-096504-B-C32 AEI, from Spanish Government.}

\begin{abstract}
We evaluate the moments of some functions composed with the fractional part of $1/x$. We name them fractional moments. In particular, we obtain expressions for the fractional moments of some trigonometric functions, the Bernoulli polynomials and the functions $x^m$ and $x^m(1-x)^m$.
\end{abstract}

\maketitle

\section{Introduction}
The main purpose of this paper is the evaluation of moments of some functions composed with the fractional part of $1/x$. In fact, we will analyze the integrals
\begin{equation*}
\label{eq:moments}
I_kf=\int_{0}^{1}x^k f\left(\bigg\{\frac{1}{x}\bigg\}\right)\, dx, \qquad k=0,1,2,\dots .
\end{equation*}
We name this integrals the fractional moments of the function $f$.

Up to our knowledge, the particular case $f(x)=x^m$ appears in the literature in different references (see \cite{Furdui-ITSF, Furdui-Libro, Furdui-Analysis, LSQ, SLQ, Valean}). For example, in \cite[Theorem 2.1]{Furdui-Analysis} or \cite[Problem 22.2]{Furdui-Libro} we can see the identity
\begin{equation}
\label{eq:Furdui-zetas}
\int_{0}^{1}x^k\left\{\frac{1}{x}\right\}^m\, dx=\frac{m!}{(k+1)!}\sum_{j=1}^{\infty}\frac{(k+j)!}{(m+j)!}(\zeta(k+j+1)-1), \qquad m,k\in \mathbb{Z}^+,
\end{equation}
where $\zeta$ denotes the zeta Riemann function. Moreover, for the particular case  $m=k+1$, by using the identity
\[
\sum_{j=1}^{\infty}\frac{\zeta(j+1)-1}{j+1}=1-\gamma,
\]
with $\gamma$ being the Euler-Mascheroni constant, it is proved that
\[
\int_{0}^{1}x^k\left\{\frac{1}{x}\right\}^{k+1}\, dx=H_{k+1}-\gamma-\sum_{j=2}^{k+1}\frac{\zeta(j)}{j},
\]
where $H_n$ is the $n$-th harmonic number. In this paper we will make a systematic analysis of this kind of integrals and we will obtain appropriated closed forms for them.

The main tool to evaluate fractional moments will be an identity relating them with an integral involving the function $\log\Gamma(x+1)$ and the derivatives of the function $f$ (see Lemma \ref{Lem:main} in the next section). From this result, we deduce the fractional moments for some functions. In particular, we analyze the fractional moments for some trigonometric functions, the Bernoulli polynomials and the functions $x^m$ and $x^m(1-x)^m$. The integrals containing $\log\Gamma(x+1)$ will be evaluated by using some results in \cite{Espinosa-Moll-1}. In that paper, we see that the integrals
\[
\int_{0}^{1}B_n(x)\log\Gamma(x)\, dx,
\]
where $B_n$ are the Bernoulli polynomials, are simpler to evaluate than the integrals
\[
\int_{0}^{1}x^n \log\Gamma(x)\, dx.
\]
In fact, the latter integrals are evaluated in terms of the former ones. By this reason, to analyze the fractional moments of $x^m(1-x)^m$ we start giving an expansion for these functions and their derivatives in terms of the Bernoulli polynomials (see Lemma \ref{Lem:expan-Berno} and the remark following it). We believe that this result has its  own interest.

The paper is organized as follows. In Section \ref{sec:main} we present the lemma allowing us to obtain the fractional moments and in the rest of the paper we show some examples and applications related to trigonometric functions, Bernoulli polynomials and the functions $x^m$ and $x^m(1-x)^m$.

All along the paper when an empty sum appears it must be taken as zero.

\section{The main lemma}
\label{sec:main}
The following lemma will be the main tool to evaluate fractional moments. Before, we define the sequence $\alpha_n=\zeta(n+1)$, for $n>0$, and $\alpha_0=\gamma$.
\begin{Lem}
\label{Lem:main}
  Let $f$ be a function having $k+2$ derivatives in the interval $[0,1]$. Then
\begin{multline*}
I_kf=\frac{1}{(k+1)!}\Bigg(\sum_{j=0}^{k}(k-j)!\left(f^{(j)}(0)\alpha_{k-j}-f^{(j)}(1)(\alpha_{k-j}-1))\right)\\
+\int_{0}^{1}f^{(k+2)}(x)\log\Gamma(x+1)\,dx\Bigg), \qquad k\ge 0.
\end{multline*}
\end{Lem}

The polygamma function will play a crucial role in the proof of the following lemma. It is defined as the $(m+1)$-th derivative of the logarithm of the gamma function
\[
\psi^{(m)}(x)=\frac{d^{m+1}}{dx^{m+1}}\log \Gamma(x), \qquad m=0,1,2,\dots .
\]
Two facts about the polygamma function will be fundamental. Its representation as a series
\begin{equation}
\label{eq:psi-series}
\psi^{(m)}(x)=(-1)^{m+1}m!\sum_{k=0}^{\infty}\frac{1}{(x+k)^{m+1}},
\end{equation}
and its values in the positive integers
\begin{align}
\psi^{(m)}(n)&=(-1)^{m+1}m!\left(\zeta(m+1)-\sum_{k=1}^{n-1}\frac{1}{k^{m+1}}\right)\notag
\\&=(-1)^{m+1}m!\sum_{k=n}^{\infty}\frac{1}{k^{m+1}},
\qquad m,n=1,2,\dots,
\label{eq:psi-integers-1}
\end{align}
and
\begin{equation}
\label{eq:psi-integers-2}
\phi^{(0)}(n)=-\gamma+\sum_{k=1}^{n-1}\frac{1}{k}, \qquad n=1,2,\dots.
\end{equation}
\begin{proof}[Proof of Lemma \ref{Lem:main}.]
With the change of variable $w=1/x$, we have
\begin{align*}
I_kf&=\int_{1}^{\infty} f(\{w\})\, \frac{dw}{w^{k+2}}=\sum_{j=1}^{\infty}\int_{j}^{j+1}f(w-j)\, \frac{dw}{w^{k+2}}\\&=\sum_{j=1}^{\infty}\int_{0}^{1}f(s)\,\frac{ds}{(j+s)^{k+2}}
=\frac{(-1)^k}{(k+1)!}\int_{0}^{1}f(s)\psi^{(k+1)}(s+1)\,ds,
\end{align*}
where in the last step we have used \eqref{eq:psi-series}. Now, applying integration by parts $k+2$ times and taking into account that $\log\Gamma(2)=\log\Gamma(1)=0$, we arrive at
\begin{multline*}
I_kf=\frac{(-1)^k}{(k+1)!}\Bigg(\sum_{j=0}^{k}(-1)^{j}(f^{(j)}(1)\psi^{(k-j)}(2)-f^{(j)}(0)\psi^{(k-j)}(1))\\
+(-1)^{k+2}\int_{0}^{1}f^{(k+2)}(s)\log \Gamma(s+1)\, ds\Bigg).
\end{multline*}
Finally, we conclude the proof applying \eqref{eq:psi-integers-1} and \eqref{eq:psi-integers-2}.
\end{proof}

\section{Fractional moments for trigonometric functions}
We start our examples given the fractional moments for the sine and cosine functions. The expression that we obtain for them involve the classical functions sine integral
\[
\Si(x)=\int_{0}^{x}\frac{\sin t}{t}\, dt
\]
and cosine integral
\[
\Ci(x)=-\int_{x}^{\infty}\frac{\cos t}{t}\, dt.
\]
The following lemma contains the evaluation of two integrals for $\log\Gamma(x+1)$ with trigonometric functions.
\begin{Lem}
\label{Lem:trig-log}
It is verified that
\[
\int_{0}^{1}\sin(2\pi x)\log\Gamma(x+1)\, dx=\frac{\Ci(2\pi)}{2\pi}
\]
and
\[
\int_{0}^{1}\cos(2\pi x)\log\Gamma(x+1)\, dx=\frac{1}{4}-\frac{\Si(2\pi)}{2\pi}.
\]
\end{Lem}

\begin{proof}
Taking $n=1$ in \cite[6.443.1 and 6.443.3]{Grad-Ri}, we have
\[
\int_{0}^{1}\sin(2\pi x)\log\Gamma(x)\, dx=\frac{\log 2\pi+\gamma}{2\pi}
\]
and
\[
\int_{0}^{1}\cos(2\pi x)\log\Gamma(x)\, dx=\frac{1}{4}.
\]
Then
\[
\int_{0}^{1}\sin(2\pi x)\log\Gamma(x+1)\, dx=\frac{\log 2\pi+\gamma}{2\pi}+\int_{0}^{1}\sin(2\pi x)\log x\, dx
\]
and
\[
\int_{0}^{1}\cos(2\pi x)\log\Gamma(x+1)\, dx=\frac{1}{4}+\int_{0}^{1}\cos(2\pi x)\log x\, dx.
\]
Now, applying integration by parts and the identity \cite[8.230.2]{Grad-Ri}
\[
\Ci(x)=\gamma+\log x+\int_{0}^{x}\frac{\cos t-1}{t}\, dt,
\]
we deduce
\begin{align*}
\int_{0}^{1}\sin(2\pi x)\log x\, dx&=\frac{1}{2\pi}\int_{0}^{1}\frac{\cos(2\pi x)-1}{x}\, dx
\\&=\frac{1}{2\pi}\int_{0}^{2\pi}\frac{\cos t-1}{t}\, dt=
-\frac{\log 2\pi+\gamma}{2\pi}+\frac{\Ci(2\pi)}{2\pi}
\end{align*}
and the result for the integral with the sine follows. The integral with the cosine can be obtained by using integration by parts only. Indeed,
\[
\int_{0}^{1}\cos(2\pi x)\log x\, dx=-\frac{1}{2\pi}\int_{0}^{1}\frac{\sin(2\pi x)}{x}\, dx
=-\frac{1}{2\pi}\int_{0}^{2\pi}\frac{\sin t}{t}\,dt=-\frac{\Si(2\pi)}{2\pi}.\qedhere
\]
\end{proof}

With the notation
\[
f_s(x)=\sin(2\pi x)\qquad\text{ and }\qquad f_c(x)=\cos(2\pi x),
\]
we have
\[
f_s^{(2j)}(x)=(-1)^j(2\pi)^{2j}\sin(2\pi x), \qquad f_s^{(2j+1)}(x)=(-1)^j(2\pi)^{2j+1}\cos(2\pi x),
\]
\[
f_c^{(2j)}(x)=(-1)^j(2\pi)^{2j}\cos(2\pi x),\quad \text{and} \quad f_c^{(2j+1)}(x)=(-1)^{j+1}(2\pi)^{2j+1}\sin(2\pi x).
\]
Moreover, $f_s^{(2j)}(0)=f_s^{(2j)}(1)=0$, $f_s^{(2j+1)}(0)=f_s^{(2j+1)}(1)=(-1)^{j}(2\pi)^{2j+1}$, $f_c^{(2j+1)}(0)=f_c^{(2j+1)}(1)=0$, and $f_c^{(2j)}(0)=f_c^{(2j)}(1)=(-1)^j(2\pi)^{2j}$.

Taking
\[
\mathcal{S}_k=\int_{0}^{1}x^k f_s\left(\left\{\frac{1}{x}\right\}\right)\, dx\qquad\text{ and }\qquad
\mathcal{C}_k=\int_{0}^{1}x^k f_c\left(\left\{\frac{1}{x}\right\}\right)\, dx
\]
we have the following result.

\begin{Theo}
For $n\ge 0$, it is verified that
\[
\mathcal{S}_{2n}=\frac{(-1)^{n+1}(2\pi)^{2n+1}}{(2n+1)!}
\left(\sum_{j=0}^{n-1}\frac{(-1)^j(2j+1)!}{(2\pi)^{2j+2}}+\Ci(2\pi)\right),
\]
\[
\mathcal{S}_{2n+1}=\frac{(-1)^n(2\pi)^{2n+2}}{(2n+2)!}
\left(\sum_{j=0}^{n}\frac{(-1)^j(2j)!}{(2\pi)^{2j+1}}-\frac{\pi}{2}+\Si(2\pi)\right),
\]
\[
\mathcal{C}_{2n}=\frac{(-1)^n(2\pi)^{2n+1}}{(2n+1)!}
\left(\sum_{j=0}^{n}\frac{(-1)^j(2j)!}{(2\pi)^{2j+1}}-\frac{\pi}{2}+\Si(2\pi)\right),
\]
and
\[
\mathcal{C}_{2n+1}=\frac{(-1)^n(2\pi)^{2n+2}}{(2n+2)!}
\left(\sum_{j=0}^{n}\frac{(-1)^j(2j+1)!}{(2\pi)^{2j+2}}+\Ci(2\pi)\right).
\]
\end{Theo}
\begin{proof}
The identities can be deduced immediately by using Lemma~\ref{Lem:main}, Lemma~\ref{Lem:trig-log}, and the given properties about the derivatives of the functions $f_s$ and $f_c$.
\end{proof}

\section{Fractional moments for Bernoulli polynomials}
Now, we analyze the fractional moments for the Bernoulli polynomials $B_n(x)$. They can be defined through its exponential generating function
\[
\frac{te^{xt}}{e^t-1}=\sum_{n=0}^{\infty}B_n(x)\frac{t^n}{n!},
\]
converging for $|t|<\pi$. It is well known that $B_n(0)=(-1)^nB_n(1)=B_n$, where $B_n$ are the Bernoulli numbers. Remember that $B_{2k+1}=0$, for $k\ge 1$, and $B_1=-1/2$. A main tool in our approach will be the identity $B_n'(x)=nB_{n-1}(x)$ (Bernoulli polynomials are, in fact, a particular case of Appell polynomials). More generally, it is verified
\begin{equation}
\label{eq:ber-der}
B_n^{(k)}(x)=\frac{n!}{(n-k)!}B_{n-k}(x),\qquad n\ge k.
\end{equation}

A crucial point to obtain a proper expression for the fractional moments of the Bernoulli polynomials is the following identity (see \cite[(6.5) and (6.6)]{Espinosa-Moll-1})
\begin{equation}
\label{eq:ber-log}
\int_{0}^{1}B_n(x)\log\Gamma(x)\, dx=a_n,\qquad n\ge 0,
\end{equation}
where the sequence $a_n$ is defined by
\begin{equation}
\label{eq:seq-an}
a_{n}=\begin{cases}
-\zeta'(-n),& n=0,2,4,\dots,\\
\frac{B_{n+1}}{n+1}\left(\frac{\zeta'(n+1)}{\zeta(n+1)}-\log(2\pi)-\gamma\right), & n=1,3,5\dots,
\end{cases}
\end{equation}
with $\zeta'$ being the derivative of the Riemann function. It is convenient to remember that $\zeta'(0)=-\log\sqrt{2\pi}$ and
\[
\zeta'(2n)=(-1)^n\frac{(2n)!\zeta(2n+1)}{2(2\pi)^{2m}}.
\]
Moreover, we consider the sequence
\[
b_n=a_n-\frac{1}{n+1}\sum_{k=1}^{n+1}\binom{n+1}{k}\frac{B_{n+1-k}}{k}, \qquad n\ge 0.
\]

\begin{Lem}
\label{Lem:Berno-Log}
For $n\ge 0$, it is verified that
\[
\int_{0}^{1}B_n(x)\log\Gamma(x+1)\, dx=b_n.
\]
\end{Lem}
\begin{proof}
From \eqref{eq:ber-log}, we have
\[
\int_{0}^{1}B_n(x)\log\Gamma(x+1)\, dx=\int_{0}^{1}B_n(x)\log x\, dx+a_n.
\]
Now, applying integration by parts to the first integral and using the identity
\begin{equation}
\label{eq:ber.ber}
B_m(x)=\sum_{k=0}^{m}\binom{m}{k}B_{m-k}x^k
\end{equation}
(which can be deduced from \eqref{eq:ber-der} by using that $B_n(0)=B_n$), we obtain that
\begin{align*}
\int_{0}^{1}B_n(x)\log x\, dx&=-\frac{1}{n+1}\int_{0}^{1}\frac{B_{n+1}(x)-B_{n+1}}{x}\, dx\\&=-\frac{1}{n+1}\sum_{k=1}^{n+1}\binom{n+1}{k}B_{n+1-k}\int_{0}^{1}x^{k-1}\, dx
\\&=-\frac{1}{n+1}\sum_{k=1}^{n+1}\binom{n+1}{k}\frac{B_{n+1-k}}{k}
\end{align*}
and the proof is completed.
\end{proof}

Taking
\[
J_k^n=\int_{0}^{1}x^kB_n\left(\left\{\frac{1}{x}\right\}\right)\, dx,
\]
we have the following result.

\begin{Theo}
For $n\ge 1$, it is verified that
\[
J_k^n=\frac{1}{(k+1)\binom{k}{n}}\Bigg(\sum_{j=0}^{[n/2]}\binom{k-n+2j}{2j}B_{2j}
+(k-n+1)\bigg(\frac{1}{2}-\zeta(k-n+2)\bigg)\Bigg),
\]
for $k\ge n$,
\[
J_{n-1}^n=\sum_{j=1}^{[n/2]}\frac{B_{2j}}{2j}+\frac{1}{2}-\gamma,
\]
and,
\begin{equation*}
J_{k}^n=\frac{1}{k+1}\binom{n}{k}\Bigg(\sum_{j=[(n-k+1)/2]}^{[n/2]}
\frac{B_{2j}}{\binom{2j}{k-n+2j}}+(n-k)(n-k-1)b_{n-k-2}\Bigg),
\end{equation*}
for $0\le k \le n-2$.
\end{Theo}

\begin{proof}
The result follows by applying Lemma \ref{Lem:main}, Lemma \ref{Lem:Berno-Log}, and the properties of Bernoulli polynomials and Bernoulli numbers.
\end{proof}

\section{Fractional moments for the functions $x^m$}
Considering the functions $h_m(x)=x^m$, with $m\ge 1$, in this section we evaluate their fractional moments. More precisely, we calculate the integrals
\[
\mathcal{C}_k^m=\int_{0}^{1}x^k h_m\left(\left\{\frac{1}{x}\right\}\right)\, dx.
\]
To this end, we need the two following lemmas.

\begin{Lem}
\label{lem:pot-log}
For $n\ge 0$, it is verified that
\[
\int_{0}^{1}x^n\log\Gamma(x+1)\, dx=-\frac{1}{(n+1)^2}+\frac{1}{n+1}\sum_{k=0}^{n}\binom{n+1}{k}a_k,
\]
where the sequence $a_k$ was defined in \eqref{eq:seq-an}.
\end{Lem}
\begin{proof}
By using the identity\footnote{This fact follows from \eqref{eq:ber.ber} applying the identity
\[
\sum_{k=0}^{m}\binom{m+1}{k}=\begin{cases}1,& m=0\\0, & m\not=0\end{cases},
\]
which can be proved by using the exponential generating function
\[
\frac{t}{e^{t}-1}=\sum_{k=0}^{\infty}\frac{B_k}{k!}t^k.
\]}
\begin{equation}
\label{eq:pot-Ber}
x^n=\frac{1}{n+1}\sum_{j=0}^{n}\binom{n+1}{j}B_j(x),
\end{equation}
it is clear that
\begin{align*}
\int_{0}^{1}x^n\log\Gamma(x+1)\, dx&=\int_{0}^{1}x^n\log x\, dx+\int_{0}^{1}x^n\log\Gamma(x)\, dx\\
&=-\frac{1}{(n+1)^2}+\frac{1}{n+1}\sum_{j=0}^{n}\binom{n+1}{j}\int_{0}^{1}B_j(x)\log\Gamma(x)\, dx
\end{align*}
and the results is obtained applying \eqref{eq:ber-log}.
\end{proof}
\begin{Lem}
\label{lem:pot-id}
The identities
\[
\sum_{j=0}^{m}\frac{(k-j)!}{(m-j)!}=\frac{(k+1)!}{m!(k+1-m)},\qquad k\ge m
\]
and
\[
\sum_{j=0}^{k}\frac{(k-j)!}{(m-j)!}=\frac{(k+1)!}{m!(k+1-m)}\left(1-\binom{m}{k+1}\right), \qquad 0\le k\le m-2,
\]
hold.
\end{Lem}
\begin{proof}
The first identity is equivalent to
\[
\sum_{j=0}^{m}\binom{k-m+j}{j}=\binom{k+1}{m}
\]
and this is a consequence of the relation
\[
\sum_{j=0}^{m}\binom{n+j}{j}=\binom{n+m+1}{m},\qquad n\ge 0,
\]
which can be proved in an elementary way by using induction on $m$.

The second identity is equivalent to
\[
\frac{m-k-1}{k+1}\sum_{j=0}^{k}\frac{\binom{m}{j}}{\binom{k}{j}}=\binom{m}{k+1}-1.
\]
In this case, it is enough to check that both sides of the identity satisfy the recurrence relation
\[
(m-k)a_{m+1,k}-(m+1)a_{m,k}=k+1, \qquad 0\le k \le m-2,
\]
and they coincide for $m=2$ and $k=0$.
\end{proof}

In the proof of our result for $\mathcal{C}_k^m$, we will apply that $h_m^{(k)}(x)=m!x^{m-k}/(m-k)!$, for $0\le k \le m$, and $h_m^{(k)}(x)=0$, for $k>m$. Moreover, $h_m^{(k)}(0)=0$, for $k\not= m$, $h_m^{(m)}(0)=m!$, and $h_m^{(k)}(1)=m!/(m-k)!$, for $0\le k \le m$.

Now, we have the following result.
\begin{Theo}
\label{moment-pot}
For $m\ge 1$, it is verified that
\[
\mathcal{C}_k^m=\frac{1}{k+1-m}-\frac{1}{(k+1)\binom{k}{m}}\sum_{j=k-m+1}^{k}\binom{j}{k-m}\zeta(j+1), \qquad k\ge m,
\]
\[
\mathcal{C}_{m-1}^m=H_m-\gamma-\sum_{j=1}^{m-1}\frac{\zeta(j+1)}{j+1},
\]
and
\begin{multline*}
\mathcal{C}_k^m=\frac{1}{k+1-m}-\frac{1}{k+1}\binom{m}{k}\left(\sum_{j=1}^{k}\frac{\zeta(j+1)}{\binom{m-k+p}{p}}
+\gamma\right)\\+\binom{m}{k+1}\sum_{j=0}^{m-k-2}\binom{m-k-1}{j}a_j, \qquad 0\le k\le m-2.
\end{multline*}
\end{Theo}
\begin{proof}
The result can de deduced by applying Lemma \ref{Lem:main}, Lemma \ref{lem:pot-log}, Lemma \ref{lem:pot-id}, and the given elementary properties for the derivatives of $h_m$.
\end{proof}
\begin{Remark}
  We have to observe that a closed form for $\mathcal{C}_k^m$, when $k\ge m$, appears in \cite[Problem 1.47]{Valean} and it is equivalent to the given one in our previous theorem.
\end{Remark}

Note that, taking $k=m$, we obtain that
\[
\mathcal{C}_m^m=1-\frac{1}{m+1}\sum_{j=1}^{m}\zeta(j+1),
\]
which is the result in \cite[Corollary 2.2]{Furdui-Analysis}. Moreover, the identity
\eqref{eq:Furdui-zetas} and the previous theorem imply the following corollary.
\begin{Cor}
  For $m\ge 1$, it is verified that
\begin{multline*}
\frac{m!}{(k+1)!}\sum_{j=1}^{\infty}\frac{(k+j)!}{(m+j)!}(\zeta(k+j+1)-1)=\frac{1}{k+1-m}
\\-\frac{1}{(k+1)\binom{k}{m}}\sum_{j=k-m+1}^{k}\binom{j}{k-m}\zeta(j+1),\qquad k\ge m,
\end{multline*}
\[
\sum_{j=1}^{\infty}\frac{\zeta(m+j)-1}{m+j}=H_m-\gamma-\sum_{j=1}^{m-1}\frac{\zeta(j+1)}{j+1},
\]
and
\begin{multline*}
\frac{m!}{(k+1)!}\sum_{j=1}^{\infty}\frac{(k+j)!}{(m+j)!}(\zeta(k+j+1)-1)=\frac{1}{k+1-m}-\frac{1}{k+1}\binom{m}{k}
\\\times \left(\sum_{j=1}^{k}\frac{\zeta(j+1)}{\binom{m-k+p}{p}}
+\gamma\right)+\binom{m}{k+1}\sum_{j=0}^{m-k-2}\binom{m-k-1}{j}a_j, \qquad 0\le k\le m-2.
\end{multline*}
\end{Cor}

By using the values
\[
a_0=\log\sqrt{2\pi}\qquad \text{ and }\qquad a_1=-\frac{1}{4}+\frac{\zeta'(2)}{2\pi^2}+\frac{1}{3}\log\sqrt{2\pi}-\frac{\gamma}{12},
\]
we can see some particular cases. Taking $k=m-2$ and $k=m-3$, we have, respectively,
\[
\sum_{j=1}^{\infty}\frac{\zeta(m+j-1)-1}{(m+j)(m+j-1)}=-\frac{1}{m}+\log\sqrt{2\pi}-\frac{\gamma}{2}-\sum_{n=2}^{m-1}\frac{\zeta(n)}{n(n+1)},\qquad m\ge 2,
\]
and
\begin{multline*}
\sum_{j=1}^{\infty}\frac{\zeta(m+j-2)-1}{(m+j)(m+j-1)(m+j-2)}\\=-\frac{1}{2m(m-1)}+\frac{\zeta'(2)}{2\pi^2}+\frac{1}{3}\log\sqrt{2\pi}
-\frac{\gamma}{4}-\sum_{n=2}^{m-2}\frac{\zeta(n)}{n(n+1)(n+2)}, \qquad m\ge 3.
\end{multline*}

\begin{Remark}
From the well-known Hermite identity
\[
\sum_{k=0}^{n-1}\left[x+\frac{k}{n}\right]=[nx], \qquad x\in \mathbb{R},
\]
we obtain that
\[
\sum_{k=0}^{n-1}\left\{x+\frac{k}{n}\right\}=\{nx\}+\frac{n-1}{2}, \qquad x\in \mathbb{R}.
\]
Then, applying Theorem \ref{moment-pot}, we can deduce that
\[
\int_{0}^{n}x^k\sum_{k=0}^{n-1}\left\{\frac{1}{x}+\frac{k}{n}\right\}\, dx=n^{k+1}\left(\frac{1}{k}+\frac{n-1-2\zeta(k+1)}{2(k+1)}\right), \qquad n,k\ge 1,
\]
and
\[
\int_{0}^{n}\sum_{k=0}^{n-1}\left\{\frac{1}{x}+\frac{k}{n}\right\}\, dx=n\left(\frac{n+1}{2}-\gamma\right), \qquad n\ge 1.
\]
\end{Remark}

\begin{Remark}
In \cite[Theorem 3.1]{Furdui-Analysis}, it is proved that
\[
\int_{0}^{1}\int_{0}^{1}\left\{\frac{x}{y}\right\}^m\left\{\frac{y}{x}\right\}^k\, dx\, dy=\frac{\mathcal{C}_k^m+\mathcal{C}_m^k}{2}.
\]
In such paper, the double integral is written as an infinite sum of values of the Riemann zeta function by using \eqref{eq:Furdui-zetas}. With Theorem \ref{moment-pot} it is possible to obtain an appropriate closed form for the integral.
\end{Remark}

\section{Fractional moments for the functions $x^m(1-x)^m$}
To finish with our examples, we study the fractional moments for the functions $f_m(x)=x^m(1-x)^m$. To do this, we start obtaining an expression for $f_m(x)$ in terms of Bernoulli polynomials because, as we said in the Introduction, the integrals of $\log\Gamma(x+1)$ with polynomials behave better with the Bernoulli polynomials $B_k(x)$ than with the usual powers $x^k$. In fact, in \cite[Example 6.4]{Espinosa-Moll-1} to evaluate the integrals
\[
\int_{0}^{1}x^k\log \Gamma(x)\, dx
\]
the authors write them in terms of Bernoulli polynomials by using \eqref{eq:pot-Ber}, as we did in Lemma \ref{lem:pot-log}. To obtain the expansion of $f_m(x)$ in terms of Bernoulli polynomials we need some sums of combinatorial numbers that are contained in the following lemma.

\begin{Lem}
\label{lem:fm-sums}
The identities
\begin{equation}
\label{eq:j0}
\sum_{k=0}^{m}\frac{(-1)^k}{m+k+1}\binom{m}{k}=\frac{1}{(2m+1)\binom{2m}{m}},
\end{equation}
\begin{equation}
\label{eq:j1-m}
\sum_{k=0}^{m}\frac{(-1)^k}{m+k+1}\binom{m}{k}\binom{m+k+1}{j}=0, \qquad j=1,\dots,m,
\end{equation}
and
\begin{equation}
\label{eq:jm+1-2m}
\sum_{k=j-m}^{m}\frac{(-1)^k}{m+k+1}\binom{m}{k}\binom{m+k+1}{j}=\frac{(-1)^m(1+(-1)^j)}{j}\binom{m}{j-m-1},
\end{equation}
for $j=m+1,\dots,2m$, hold
\end{Lem}
\begin{proof}
To prove \eqref{eq:j0} we apply integration. Indeed, it is easy to check that
\begin{align*}
\sum_{k=0}^{m}\frac{(-1)^k}{m+k+1}\binom{m}{k}&=\sum_{k=0}^{m}(-1)^k\binom{m}{k}\int_{0}^{1}x^{m+k}\, dx
\\&=\int_{0}^{1}x^m(1-x)^m\, dx\\&=\frac{(\Gamma(m+1))^2}{\Gamma(2m+2)}=\frac{1}{(2m+1)\binom{2m}{m}}.
\end{align*}

In the proof of \eqref{eq:j1-m} we use the falling factorial, and it is given by $(a)_n=a(a-1)\cdots (a-n+1)$, for $n>1$, and $(a)_0=1$. It is clear that $\binom{m+k+1}{j}$ is a polynomial in the variable $k$ of degree $j$ and, for $j\ge 1$,
\begin{align*}
\binom{m+k+1}{j}\frac{1}{m+k+1}&=\frac{(m+k)(m+k-1)\cdots (m+k+2-j)}{j!}\\&=\sum_{\ell=0}^{j-1}a_{j,m,\ell}(k)_\ell
\end{align*}
for some coefficients $a_{j,m,\ell}$. From the identity
\[
\binom{m}{k}(k)_\ell=\binom{m-\ell}{k-\ell}(m)_\ell,
\]
we have
\begin{multline*}
\sum_{k=0}^{m}\frac{(-1)^k}{m+k+1}\binom{m}{k}\binom{m+k+1}{j}=\sum_{\ell=0}^{j-1}a_{j,m,\ell}
\sum_{k=\ell}^{m}(-1)^k\binom{m}{k}(k)_\ell\\
\begin{aligned}
&=\sum_{\ell=0}^{j-1}a_{j,m,\ell}\sum_{k=\ell}^{m}(-1)^k\binom{m-\ell}{k-\ell}(m)_\ell\\&=
\sum_{\ell=0}^{j-1}(-1)^\ell a_{j,m,\ell}(m)_\ell\sum_{k=0}^{m-\ell}(-1)^k\binom{m-\ell}{k}=0,
\end{aligned}
\end{multline*}
where in the last step we have applied the identity $\sum_{k=0}^{p}(-1)^k\binom{p}{k}=0$ with $p=m-j+1,\dots, m$.

The identity \eqref{eq:jm+1-2m} is equivalent to
\[
\sum_{k=j-m}^{m}(-1)^k\binom{m}{k}\binom{m+k}{j-1}=(-1)^m(1+(-1)^j)\binom{m}{j-m-1}.
\]
Obviously,
\begin{multline*}
\sum_{k=j-m}^{m}(-1)^k\binom{m}{k}\binom{m+k}{j-1}\\=\sum_{k=j-m-1}^{m}(-1)^k\binom{m}{k}\binom{m+k}{j-1}-(-1)^{j-m-1}\binom{m}{j-m-1}.
\end{multline*}
Now, applying \cite[(5.24)]{Concrete}, we have
\[
\sum_{k=j-m-1}^{m}(-1)^k\binom{m}{k}\binom{m+k}{j-1}=(-1)^m\binom{m}{j-m-1}
\]
and
\[
\sum_{k=j-m}^{m}(-1)^k\binom{m}{k}\binom{m+k}{j-1}=(-1)^m(1+(-1)^j)\binom{m}{j-m-1}
\]
finishing the proof.
\end{proof}

\begin{Lem}
\label{Lem:expan-Berno}
For each $m\ge 1$, it is verified
\[
f_m(x)=\frac{1}{(2m+1)\binom{2m}{m}}+(-1)^m\sum_{j=m+1}^{2m}\frac{1+(-1)^j}{j}\binom{m}{j-m-1}B_j(x).
\]
\end{Lem}

\begin{proof}
To obtain the identity, we use Newton's binomial identity and \eqref{eq:pot-Ber}. Indeed,
\begin{align*}
f_m(x)&=x^m\sum_{k=0}^{m}(-1)^k\binom{m}{k}x^k=\sum_{k=0}^{m}(-1)^k\binom{m}{k}x^{m+k}\\&=\sum_{k=0}^{m}\frac{(-1)^k}{m+k+1}\binom{m}{k}\sum_{j=0}^{m+k}\binom{m+k+1}{j}B_j(x)
\\&=\sum_{j=0}^{m}B_j(x)\sum_{k=0}^{m}\frac{(-1)^k}{m+k+1}\binom{m}{k}\binom{m+k+1}{j}
\\&\kern25 pt+\sum_{j=m+1}^{2m}B_j(x)\sum_{k=j-m}^{m}\frac{(-1)^k}{m+k+1}\binom{m}{k}\binom{m+k+1}{j}
\\&=\sum_{k=0}^{m}\frac{(-1)^k}{m+k+1}\binom{m}{k}+\sum_{j=1}^{m}B_j(x)\sum_{k=0}^{m}\frac{(-1)^k}{m+k+1}\binom{m}{k}\binom{m+k+1}{j}
\\&\kern25 pt+\sum_{j=m+1}^{2m}B_j(x)\sum_{k=j-m}^{m}\frac{(-1)^k}{m+k+1}\binom{m}{k}\binom{m+k+1}{j}
\end{align*}
We finish applying the identities \eqref{eq:j0}, \eqref{eq:j1-m}, and \eqref{eq:jm+1-2m} in the previous lemma.
\end{proof}

\begin{Remark}
It is interesting to observe that, by using \eqref{eq:ber-der}, we have
\begin{align}
\label{eq:derivatives}
f_m^{(k)}(x)&=(-1)^m\sum_{j=\max\{k,m+1\}}^{2m}\frac{1+(-1)^j}{j}\binom{m}{j-m-1}\frac{j!}{(j-k)!}B_{j-k}(x),\notag
\\&=(-1)^m k!\sum_{j=\max\{k,m+1\}}^{2m}\frac{1+(-1)^j}{j}\binom{m}{j-m-1}\binom{j}{k}B_{j-k}(x),
\end{align}
for $1\le k\le 2m$.
\end{Remark}

\begin{Remark}
Let $\mathcal{P}_m$ be the shifted Legendre polynomial of degree $m$; i. e, $\mathcal{P}_{m}(x)=P_m(2x-1)$, where $P_m$ is the standard Legendre polynomial. By Rodrigues' formula
\[
\mathcal{P}_m(x)=\frac{(-1)^m}{m!}f_m^{(m)}(x)
\]
and applying \eqref{eq:derivatives} we deduce that, for $m\ge 1$,
\begin{align*}
\mathcal{P}_m(x)&=\frac{1}{m!}\sum_{j=m+1}^{2m}\frac{1+(-1)^j}{j}\binom{m}{j-m-1}\frac{j!}{(j-m)!}B_{j-m}(x)\\
&=\sum_{j=1}^{m}(1+(-1)^{m+j})\frac{(m+j-1)!}{j!\, (j-1)!\, (m-j+1)!}B_j(x).
\end{align*}
The previous identity for shifted Legendre polynomials matches with the given ones in \cite[Theorem 2.2 and Theorem 2.4]{Navas-Ruiz-Varona}.
\end{Remark}

To evaluate the fractional moments of the functions $f_m$ we need the combinatorial identities contained in the next lemma.
\begin{Lem}
\label{lem:im-sums}
For $m\ge 1$, the identities
\begin{equation}
\label{eq:im-sum-1}
\sum_{j=m}^{2m}\frac{\binom{m}{j-m}}{\binom{k}{j}}=\frac{m!(k-2m)!(k+1)}{(k+1-m)!}, \qquad k\ge 2m,
\end{equation}
and
\begin{equation}
\label{eq:im-sum-2}
\sum_{j=m}^{2m-1}\frac{\binom{m}{j-m}}{\binom{2m-1}{j}}=2m(H_{2m}-H_m)
\end{equation}
hold.
\end{Lem}

\begin{proof}
The first identity is equivalent to
\[
\sum_{n=0}^{m}\binom{m+n}{m}\binom{k-m-n}{k-2m}=\binom{k+1}{m}
\]
and it follows from \cite[(5.26)]{Concrete}. To prove the second one we check that both sides satisfy the recurrence relation
\[
ma_{m+1}=(m+1)a_m+\frac{m}{2m+1}
\]
and they coincide for $m=1$.
\end{proof}

To complete the evaluation of the fractional moments for the functions $f_m$, we will need some elementary facts about them. More exactly, $f_m^{(j)}(x)=0$, for $j>2m$, $f_m^{(j)}(0)=f_m^{(j)}(1)=0$, for $0\le j <m$, and, by using Leibniz rule for the derivative of a product,
\begin{equation}
\label{eq:derivative-0-1}
(-1)^jf^{(j)}_m(0)=f_m^{(j)}(1)=(-1)^mj!\binom{m}{j-m}, \qquad m\le j \le 2m,
\end{equation}

Then, denoting
\[
  \mathcal{I}_k^m=\int_{0}^{1}x^k f_m\left(\left\{\frac{1}{x}\right\}\right)\, dx,
\]
\[
\delta(x,y)=\begin{cases}1, & x=y,\\ 0, & x\not = y,\end{cases}
\qquad \text{ and } \qquad
p_{m,k}=\sum_{j=m}^{k}\frac{\binom{m}{j-m}}{\binom{k}{j}},
\]
we have the following result.
\begin{Theo}
  For $m\ge 1$, it is verified that
  \[
  \mathcal{I}_k^m=(-1)^m\Bigg(\frac{1}{(k+1-m)\binom{k-m}{m}}-\frac{2}{k+1}\sum_{j=[m/2]}^{m-1}
  \frac{\binom{m}{2j+1-m}}{\binom{k}{2j+1}}\zeta(k-2j)\Bigg),
  \]
for $k\ge 2m$,
\[
\mathcal{I}_{2m-1}^m=(-1)^m\Bigg(H_{2m}-H_m-\gamma-\frac{1}{m}\sum_{j=[m/2]}^{m-2}
\frac{\binom{m}{2j+1-m}}{\binom{2m-1}{2j+1}}\zeta(2m-1-2j)\Bigg),
\]
\begin{multline*}
\mathcal{I}_k^m=
\frac{(-1)^m}{k+1}\Bigg(p_{m,k}-2\sum_{j=[m/2]}^{[k/2]-1}\frac{\binom{m}{2j+1-m}}{\binom{k}{2j+1}}\zeta(k-2j)\\
-2\gamma\binom{m}{2m-k}\delta((k+1)/2,[(k+1)/2])\Bigg)\\
+(-1)^m(k+2)\sum_{[m/2]+1}^{m}\binom{m}{2j-m-1}\binom{2j}{k+2}\frac{b_{2j-k-2}}{j},
\end{multline*}
for $m\le k \le 2m-2$, and
\[
\mathcal{I}_k^m=(-1)^m(k+2)\sum_{j=[m/2]+1}^{m}\binom{m}{2j-m-1}\binom{2j}{k+2}\frac{b_{2j-k-2}}{j}, 
\]
for $0\le k \le m-1$.
\end{Theo}

\begin{proof}
The result is a consequence of Lemma \ref{Lem:main}, \eqref{eq:derivatives}, Lemma \ref{Lem:Berno-Log}, and Lemma \ref{lem:im-sums}. In particular, we have to use \eqref{eq:im-sum-1} for the case $k\ge 2m$ and \eqref{eq:im-sum-2} for $k=2m-1$. Moreover, the values of the derivatives of $f_m$ in $x=0$ and $x=1$ given in \eqref{eq:derivative-0-1} are also used.
\end{proof}

\begin{Remark}
Unfortunately, we could not find a closed form for the sum
\[
p_{m,k}=\sum_{j=m}^{k}\frac{\binom{m}{j-m}}{\binom{k}{j}}
\]
appearing in the case $m\le k \le 2m-2$. This remains as an open question.

\end{Remark}


\end{document}